\documentclass[preprint,12pt]{elsarticle}

\usepackage[utf8]{inputenc}
\usepackage[T1]{fontenc}
\usepackage{amsmath,amssymb,amsthm,mathtools}
\usepackage{enumitem}
\usepackage{array}
\usepackage{hyperref}

\hypersetup{hidelinks}

\theoremstyle{plain}
\newtheorem{theorem}{Theorem}[section]
\newtheorem{lemma}[theorem]{Lemma}
\newtheorem{proposition}[theorem]{Proposition}
\newtheorem{corollary}[theorem]{Corollary}

\newtheorem{problem}[theorem]{Problem}

\theoremstyle{definition}
\newtheorem{definition}[theorem]{Definition}

\newtheorem{counterexample}[theorem]{Counterexample}

\theoremstyle{remark}
\newtheorem{remark}[theorem]{Remark}

\newcommand{\R}{\mathbb{R}}
\newcommand{\N}{\mathbb{N}}

\newcommand{\Lap}{\mathcal{L}}
\newcommand{\CM}{\mathrm{CM}}
\newcommand{\BF}{\mathrm{BF}}
\newcommand{\CBF}{\mathrm{CBF}}
\newcommand{\LCM}{\mathrm{LCM}}
\newcommand{\St}{\mathcal{S}}
\newcommand{\LP}{\mathcal{LP}}
\newcommand{\WL}{\mathsf{W}_{\!L}}
\newcommand{\WS}{\mathsf{W}_{\!S}}
\newcommand{\WH}{\mathsf{W}_{\!H}}
\DeclareMathOperator{\arsinh}{arsinh}
\DeclareMathOperator{\arcosh}{arcosh}
\DeclareMathOperator{\Imag}{Im}

\begin{document}

\begin{frontmatter}

\title{Determinacy Witnesses in the Completely\\
Monotone--Stieltjes--Bernstein Hierarchy}
\author[addr1]{Domingos S. P. Salazar\corref{cor1}}
\ead{domingos.salazar@ufrpe.br}
\address[addr1]{Unidade de Educa\c{c}\~ao a Dist\^ancia e Tecnologia,
Universidade Federal Rural de Pernambuco,
52171-900 Recife, Pernambuco, Brazil}
\cortext[cor1]{Corresponding author.}

\begin{abstract}
Membership problems in the completely monotone--Stieltjes--Bernstein
hierarchy often become explicit after passage to a determinate
representation. We organize that passage as a witness calculus. For the
Laplace, Stieltjes, and Hausdorff transforms used here, \(\mathsf W f\)
denotes the unique admissible signed representing datum; a certificate either
proves its positivity or exhibits an obstruction. A concrete dictionary for
completely monotone, logarithmically completely monotone, Bernstein, complete
Bernstein, and Stieltjes functions and for Hausdorff moment sequences is
combined with transport rules for products, shifts, sampling, and Bernstein
increments.

The applications are consequences of this common calculation. A
triple-product identity for a digamma bridge controls ordinary Bernstein and
logarithmic complete monotonicity ranges for gamma ratios and gives the strict
power cutoff in Szab\'o's problem. For integer gamma quotients, the inverse
Laplace density is a damped Jacobi polynomial, so classical zero geometry
decides broad parameter regions. Boundary signs distinguish Bernstein from
complete Bernstein behavior and separate completely monotone functions from
the Stieltjes class. On the discrete side, signed atoms and finite differences
settle interpolation and coefficient questions. The Ramanujan integral and a
fractional Volterra symbol provide positive recognition examples. Across the
applications, the workflow is the same: normalize the class question, compute
or transport its witness, and decide one explicit sign problem. Every
exclusion is accompanied by an exact certificate.
\end{abstract}


\end{frontmatter}

\section{Introduction}

Complete monotonicity packages infinitely many derivative inequalities into
one class condition. The Bernstein--Widder theorem replaces those inequalities
by positivity of a unique Laplace measure; Stieltjes, Bernstein, complete
Bernstein, and Hausdorff moment problems have parallel determinate
representations \cite{bernstein,hausdorff,widder,schilling-song-vondracek}.
For structured functions, the representation can be easier to sign than the
initial derivative tower.

Inverse-transform proofs have a long history in special-function inequalities.
Ismail, for example, uses inverse Laplace and Stieltjes transforms to prove
complete monotonicity for quotients of Bessel functions
\cite{ismail-bessel}. Integral representations have also been formulated as
positivity certificates in a broader algebraic setting
\cite{kozhasov-michalek-sturmfels}. We organize three classical determinacy
principles into one proof architecture for the completely
monotone--Stieltjes--Bernstein hierarchy and its discrete boundary. The
abstract implication is elementary. The contribution lies in explicit
witness formulas, transport rules, and exact sign certificates for the
applications below.

The architecture has three stages:
\begin{equation}\label{eq:workflow}
\text{reduce}
\ \longrightarrow\
\text{compute }\mathsf W
\ \longrightarrow\
\text{sign }\mathsf W.
\end{equation}
Here a \emph{witness} is the full inverse datum in an injective transform. A
\emph{certificate} is the sign argument applied to that datum or to an
equivalent analytic criterion. This distinction allows a negative atom or
density to be treated differently from a forbidden Pick boundary sign. The
preprocessing maps that make the main applications natural are summarized in
Table~\ref{tab:dictionary} and formalized in Section~\ref{sec:prelim}.

\begin{table}[t]
\centering
\caption{Class questions and their canonical reductions.}
\label{tab:dictionary}
\footnotesize
\begin{tabular}{@{}>{\raggedright\arraybackslash}p{0.22\linewidth}
>{\raggedright\arraybackslash}p{0.31\linewidth}
>{\raggedright\arraybackslash}p{0.38\linewidth}@{}}
\hline
Question & Preprocessing & Decisive datum \\
\hline
\(f\in\CM\) & none & Laplace witness \(\WL[f]\) \\
\(f\in\LCM\), \(g\in\BF\) &
\(- (\log f)'\), respectively \(g'\) & corresponding Laplace witness \\
\(f\in\St\), \(g\in\CBF\) &
\(f\), respectively \(1/g\) & Stieltjes witness \(\WS\); Pick-sign certificate \\
Hausdorff sequence & moments or finite differences &
Hausdorff witness \(\WH\); difference table \\
Bernstein interpolation & unit increments & Hausdorff witness of the increments \\
\hline
\end{tabular}
\end{table}

The first application family is governed by the digamma bridge
\begin{equation}\label{eq:intro-bridge}
H_\alpha(x)=\psi(x+\alpha)-\psi(x)-\frac{\alpha}{x}
=(\log F_\alpha)'(x),
\qquad
F_\alpha(x)=\frac{\Gamma(x+\alpha)}{x^\alpha\Gamma(x)}.
\end{equation}
Its triple-product expansion fixes the sign of \(H_\alpha\) for every
\(\alpha>0\). A complement measure then gives the exact range
\(F_\alpha\in\BF\) if and only if \(0\le\alpha\le1\), answering Simon's
question \cite{simon-moment}. Simon already records that the nontrivial range
lies outside \(\CBF\) \cite[p.~591]{simon-moment}; we add a direct
boundary-phase certificate that displays the obstruction. The same bridge
gives an ordinary \(\LCM/\CM\) proof for the reciprocal ratio and the strict
cutoff in Szab\'o's power-weight problem \cite{szabo-cm}. Recent work of
Koumandos and Pedersen places the reciprocal ratio in stronger generalized
Bernstein and generalized Stieltjes frameworks
\cite{koumandos-pedersen-two-parameter,
koumandos-pedersen-approximations}; the role of the present bridge is to
compress the ordinary classifications and their certificates.

The second family starts from the symmetric gamma quotient
\begin{equation}\label{eq:intro-gab}
g_{a,b}(y)=
\frac{\Gamma(y+a)\Gamma(y+b)}
{\Gamma(y)\Gamma(y+a+b)}.
\end{equation}
For integer \(a\), the inverse Laplace transform of \(g_{a,b}'\) is
\begin{equation}\label{eq:intro-jacobi}
(-1)^{a+1}b\,t e^{-bt}
P_{a-1}^{(b-a,1)}(1-2e^{-t}).
\end{equation}
The Bernstein question of Baricz \cite{baricz-turan} is thereby reduced to
Jacobi-polynomial sign geometry. Orthogonal zeros settle \(b>a-1\), pole--zero
interlacing settles \(0<b\le1\), and the tail sign excludes odd intermediate
windows.

The remaining applications test the other two channels. A Pick argument
upgrades Wakrim's sharp Bernstein range for a fractional Volterra symbol to
the complete Bernstein class \cite{wakrim-w}; an opposite boundary sign
separates a logarithmic derivative built from \(\arsinh^2\sqrt{x}\) from the
Stieltjes cone \cite{jedidi-vakeroudis}. A sine mixture recognizes the
Ramanujan integral as Stieltjes and its correctly normalized primitive as
complete Bernstein \cite{hardy,mishra-swaminathan}. On the discrete side,
signed geometric atoms classify Bernstein interpolation for geometrically
regular weighted-shift sequences \cite{benhida-curto-exner}; an exact seventh
difference settles the Berg--Pedersen Hausdorff question
\cite{berg-pedersen-horn}; and a Hausdorff representation of exponential
generating functions gives rigidity for Laguerre--P\'olya coefficient
sequences \cite{wang-yang}. Restrictions of completely monotone functions to
integer lattices and their converses have also been studied systematically by
Aguech and Jedidi \cite{aguech-jedidi}.

Section~\ref{sec:prelim} develops the witness dictionary and its transport
laws. Sections~\ref{sec:bridge} and
\ref{sec:quotient-obstructions} treat Laplace witnesses for gamma quotients.
Sections~\ref{sec:symbols} and~\ref{sec:stieltjes} treat Pick and Stieltjes
certificates. Section~\ref{sec:discrete} treats Hausdorff witnesses. The final
section records the limitations of the method and two open witness-sign
problems.

\section{Determinacy witnesses and transport}\label{sec:prelim}

We begin with the standard classes and then isolate the common logical step.
The representation theory is classical; see
\cite{widder,schilling-song-vondracek}.

\begin{definition}[The hierarchy]\label{def:classes}
Let all functions below be defined on \((0,\infty)\) and take values in
\([0,\infty)\).
\begin{enumerate}[label=\textup{(\roman*)},leftmargin=2.3em]
\item \(f\) is \emph{completely monotone} \((f\in\CM)\) if
\(f\in C^\infty\) and \((-1)^nf^{(n)}(x)\ge0\) for every \(n\ge0\) and
\(x>0\).
\item \(f\) is a \emph{Stieltjes function}, written \(f\in\St\), if
\begin{equation}\label{eq:stieltjes-representation}
f(x)=\frac{a}{x}+b+\int_{(0,\infty)}\frac{\sigma(dt)}{x+t},
\qquad
a,b\ge0,
\quad
\int_{(0,\infty)}\frac{\sigma(dt)}{1+t}<\infty.
\end{equation}
\item \(g\) is a \emph{Bernstein function}, written \(g\in\BF\), if
\(g\in C^\infty\) and \(g'\in\CM\). Equivalently, its unique
L\'evy--Khintchine triplet satisfies
\begin{equation}\label{eq:levy-khintchine}
g(x)=a+bx+\int_{(0,\infty)}(1-e^{-xt})\,\nu(dt),
\qquad
a,b\ge0,
\quad
\int(1\wedge t)\,\nu(dt)<\infty
\end{equation}
\cite[Theorem~3.2]{schilling-song-vondracek}. We use the finite boundary
value \(g(0):=g(0^+)=a\) whenever integer interpolation includes \(0\).
\item A Bernstein function is a \emph{complete Bernstein function}, written
\(g\in\CBF\), if its L\'evy measure is absolutely continuous with a
completely monotone density.
\item A positive \(f\) is \emph{logarithmically completely monotone}, written
\(f\in\LCM\), if \(- (\log f)'\in\CM\). One has
\(\LCM\subsetneq\CM\) \cite{qi-chen,berg-mediterr}; \(\LCM\) is Horn's
class of infinitely divisible completely monotone functions
\cite{horn,berg-mediterr}.
\end{enumerate}
A sequence \((c_n)_{n\ge0}\) is \emph{completely monotone} if
\(\Delta^kc_n\ge0\) for all \(n,k\ge0\), where
\(\Delta c_n=c_n-c_{n+1}\).
\end{definition}

\begin{definition}[Determinate witness]\label{def:witness}
Let \(\mathcal E\) be a real vector space of admissible signed data,
\(\mathcal E_+\) its positive cone, and
\(T:\mathcal E\to\mathcal F\) an injective linear map. For
\(f\in T(\mathcal E)\), define its \emph{\(T\)-witness} by
\begin{equation}\label{eq:abstract-witness}
\mathsf W_T[f]:=T^{-1}f.
\end{equation}
A \emph{certificate} is a verifiable argument that
\(\mathsf W_T[f]\) belongs to \(\mathcal E_+\) or fails to belong to it.
\end{definition}

Injectivity gives the elementary identity
\begin{equation}\label{eq:determinacy-principle}
f\in T(\mathcal E_+)
\quad\Longleftrightarrow\quad
\mathsf W_T[f]\in\mathcal E_+.
\end{equation}
The mathematical work begins after this reduction: one must compute the
witness in the admissible signed class and make its sign easier than the
initial class condition.

\begin{proposition}[The concrete witness dictionary]\label{prop:witness}
The following determinate transforms will be used.
\begin{enumerate}[label=\textup{(\roman*)},leftmargin=2.3em]
\item Let \(\mathcal M_L\) be the signed Borel measures \(\mu\) on
\([0,\infty)\) satisfying
\begin{equation}\label{eq:laplace-admissibility}
\int_{[0,\infty)}e^{-xt}\,d|\mu|(t)<\infty
\qquad(x>0).
\end{equation}
If
\begin{equation}\label{eq:laplace-witness}
f(x)=\int_{[0,\infty)}e^{-xt}\,d\mu(t),
\end{equation}
then \(\WL[f]:=\mu\) is unique in \(\mathcal M_L\), and
\(f\in\CM\) if and only if \(\WL[f]\ge0\).

\item If
\begin{equation}\label{eq:hausdorff-map}
a_n=\int_{[0,1]}s^n\,d\rho(s)
\end{equation}
for a finite signed measure \(\rho\), then \(\WH[a]:=\rho\) is unique. A
sequence is a Hausdorff moment sequence if and only if its difference table
\begin{equation}\label{eq:hausdorff-witness}
D_a(m,n):=\Delta^ma_n
=\sum_{j=0}^{m}(-1)^j\binom mj a_{n+j}\ge0
\qquad(m,n\ge0).
\end{equation}
When \(\rho\) exists,
\(D_a(m,n)=\int s^n(1-s)^m\,d\rho(s)\).

\item The data \((a,b,\sigma)\) in
\eqref{eq:stieltjes-representation} are unique; write
\(\WS[f]=(a,b,\sigma)\), with componentwise positivity. Equivalently,
\(f\in\St\) exactly when \(f\ge0\) on \((0,\infty)\) and \(f\) extends
holomorphically to \(\mathbb C\setminus(-\infty,0]\) with
\(\Imag f(z)\le0\) for \(\Imag z>0\). The measure is recovered in the vague
Stieltjes--Perron sense from
\(-\pi^{-1}\Imag f(-t+i\varepsilon)\,dt\).
\end{enumerate}
For witnesses in the signed classes above, the dictionary is
\begin{equation}\label{eq:witness-dictionary}
\begin{array}{rcl}
f\in\CM
&\Longleftrightarrow& \WL[f]\ge0,\\
f\in\LCM
&\Longleftrightarrow& f>0\ \text{and}\ \WL[-(\log f)']\ge0,\\
g\in\BF
&\Longleftrightarrow& g\ge0\ \text{and}\ \WL[g']\ge0,\\
f\in\St
&\Longleftrightarrow& \WS[f]\ge0,\\
g\in\CBF\setminus\{0\}
&\Longleftrightarrow& g>0\ \text{and}\ \WS[1/g]\ge0.
\end{array}
\end{equation}
For the triplet in \eqref{eq:levy-khintchine}, the Bernstein witness and the
L\'evy measure are related by
\begin{equation}\label{eq:bf-witness-versus-levy}
\WL[g']=b\,\delta_0+t\,\nu(dt).
\end{equation}
\end{proposition}

\begin{proof}
The positive Laplace assertion is the Bernstein--Widder theorem
\cite{bernstein,widder}. For signed injectivity, suppose \(\Lap\{\mu\}=0\)
and fix \(x_0>0\). The tilt \(e^{-x_0t}\mu(dt)\) is a finite signed measure.
Its pushforward under \(t\mapsto e^{-t}\) has every moment equal to zero, so
uniform density of polynomials in \(C[0,1]\) makes that pushforward, and hence
\(\mu\), zero. Hausdorff's theorem gives the difference criterion
\cite{hausdorff}; the same polynomial-density argument gives signed
determinacy \cite{shohat-tamarkin}. The Stieltjes representation, its inversion, the
Pick characterization, and the reciprocal equivalence with complete
Bernstein functions are in
\cite[Chapters~6 and~7]{schilling-song-vondracek}; see also
\cite{donoghue}. Differentiating \eqref{eq:levy-khintchine} gives
\eqref{eq:bf-witness-versus-levy}.
\end{proof}

\begin{remark}[Witnesses and analytic certificates]\label{rem:calculus}
The symbols \(\WL\), \(\WS\), and \(\WH\) denote full inverse data. A
pointwise cut value is generally a certificate instead: singular components
of a Stieltjes measure need not appear in a pointwise boundary density. One
value with the wrong upper-half-plane sign still excludes Stieltjes or
complete Bernstein membership through the corresponding Pick
characterization.
\end{remark}

The notation becomes useful because standard operations act transparently on
the witnesses.

\begin{proposition}[Transport laws]\label{prop:transport}
\begin{enumerate}[label=\textup{(\roman*)},leftmargin=2.3em]
\item If \(\WL[f]=\mu\), \(\WL[h]=\nu\), \(r,s\in\R\), \(c\ge0\), and
\(\lambda>0\), then, whenever the expressions are defined,
\begin{equation}\label{eq:laplace-transport}
\begin{aligned}
\WL[rf+sh]&=r\mu+s\nu,
&\WL[fh]&=\mu*\nu,\\
\WL[f(\,\cdot+c)]&=e^{-ct}\mu(dt),
&\WL[f(\lambda\,\cdot)]&=(t\mapsto\lambda t)_\#\mu.
\end{aligned}
\end{equation}
Thus \(\CM\) is closed under products, positive linear combinations, and
positive mixtures. If \(f\in\CM\) and \(g\in\BF\), then
\(f\circ g\in\CM\); the reciprocal of a positive Bernstein function is
completely monotone.

\item If \(\phi=\Lap\{\mu\}\in\CM\) and
\(F(x)=\int_0^\infty e^{-xt}\phi(t)\,dt<\infty\), then
\begin{equation}\label{eq:stieltjes-lift}
F(x)=\int_{[0,\infty)}\frac{\mu(ds)}{x+s}\in\St.
\end{equation}

\item One has
\(f\in\St\setminus\{0\}\) if and only if
\(1/f\in\CBF\setminus\{0\}\). Moreover \(\St\subset\CM\), and
\(g\in\CBF\) if and only if \(g\ge0\) on \((0,\infty)\) and \(g\) extends
holomorphically to \(\mathbb C\setminus(-\infty,0]\) with
\(\Imag g(z)\ge0\) for \(\Imag z>0\).

\item If \(\WL[f]=\mu\) and \(\lambda>0\), then the sampled sequence
\(a_n=f(n+\lambda)\) has Hausdorff witness
\begin{equation}\label{eq:laplace-to-hausdorff}
\WH[a]=(t\mapsto e^{-t})_\#\bigl(e^{-\lambda t}\mu(dt)\bigr).
\end{equation}
\end{enumerate}
\end{proposition}

\begin{proof}
The identities in \textup{(i)} follow by linearity, convolution, and change
of variables; the closure, composition, and reciprocal statements are
standard \cite[Chapters~1 and~3]{schilling-song-vondracek}. For
\textup{(ii)}, nonnegative integrands permit reversal of integration and give
\eqref{eq:stieltjes-lift}. Part \textup{(iii)} follows from the reciprocal and
Pick characterizations
\cite[Theorems~6.7 and~7.3]{schilling-song-vondracek}. Finally,
\begin{equation}
a_n=\int_{[0,\infty)}(e^{-t})^n e^{-\lambda t}\,d\mu(t),
\end{equation}
which proves \textup{(iv)}.
\end{proof}

\begin{lemma}[Local sign certificates]\label{lem:one-point}
\begin{enumerate}[label=\textup{(\roman*)},leftmargin=2.3em]
\item A completely monotone function is nonnegative, nonincreasing, and
convex. A certified point where \(f<0\), \(f'>0\), or \(f''<0\) excludes
\(f\in\CM\); a violation of complete monotonicity by \(g'\) excludes
\(g\in\BF\).
\item If \(f\ge0\) is nonincreasing and \(f(0^+)=0\), then \(f\equiv0\).
Hence a nonzero completely monotone function cannot vanish at \(0^+\).
\end{enumerate}
\end{lemma}

\begin{proof}
Part \textup{(i)} restates the defining signs. For \textup{(ii)}, if
\(0<y<x\), then \(f(x)\le f(y)\to0\) as \(y\downarrow0\).
\end{proof}

\begin{lemma}[Bernstein increments are Hausdorff]\label{lem:bf-increments}
Let \(F\in\BF\), set \(F(0)=F(0^+)\), and define
\(d_n=F(n+1)-F(n)\). If \((a,b,\nu)\) is the triplet in
\eqref{eq:levy-khintchine}, then
\begin{equation}\label{eq:bf-increment-witness}
\WH[d]
=b\,\delta_1
+(t\mapsto e^{-t})_\#\bigl((1-e^{-t})\nu(dt)\bigr).
\end{equation}
In particular, \((d_n)_{n\ge0}\) is a Hausdorff moment sequence.
\end{lemma}

\begin{proof}
The L\'evy--Khintchine representation gives
\begin{equation}
d_n=b+\int_{(0,\infty)}e^{-nt}(1-e^{-t})\,\nu(dt).
\end{equation}
The measure in \eqref{eq:bf-increment-witness} is finite because
\(1-e^{-t}\asymp1\wedge t\), and its moments are \((d_n)\).
\end{proof}

\section{The digamma bridge: one kernel, four classifications}
\label{sec:bridge}

For \(\alpha>0\) define, on \(x>0\),
\begin{equation}\label{eq:H-def}
H_\alpha(x)=\psi(x+\alpha)-\psi(x)-\frac{\alpha}{x},
\qquad
F_\alpha(x)=\frac{\Gamma(x+\alpha)}{x^\alpha\,\Gamma(x)},
\end{equation}
where \(\psi=\Gamma'/\Gamma\), so that \(H_\alpha=(\log F_\alpha)'\). The
quotient \(F_\alpha\) is classical: Wendel \cite{wendel} proved
\(F_\alpha(x)\to1\) as \(x\to\infty\), the associated two-sided bounds are
the Gautschi--Kershaw inequalities \cite{gautschi,kershaw}, and the
extensive literature on ratios of gamma functions is surveyed in
\cite{qi-ratio}. Simon \cite{simon-moment} raised the question whether
\(F_\alpha\) is a Bernstein function for \(0<\alpha<1\). Independently,
Szab\'o \cite{szabo-cm} asked for the exact cutoff \(\alpha_0\) such that
\(y^aH_d(y)\) is strictly completely monotone precisely for
\(a\le\alpha_0\), given \(0<d<1\). These questions and the sign-adjusted
extension to \(d>1\) are governed by the following expansion.

\begin{theorem}[Triple-product expansion of the digamma bridge]
\label{thm:identity}
For every \(\alpha>0\) and \(x>0\),
\begin{equation}\label{eq:identity}
H_\alpha(x)
=\alpha(1-\alpha)\sum_{k=0}^{\infty}
\frac{1}{(x+k)(x+k+1)(x+k+\alpha)} ,
\end{equation}
with locally uniform absolute convergence of the series and of all its
termwise derivatives. Consequently:
\begin{enumerate}[label=\textup{(\roman*)},leftmargin=2.3em]
\item for \(0<\alpha<1\), \(H_\alpha\) is strictly completely monotone;
\item \(H_1\equiv0\), and for \(\alpha>1\), \(-H_\alpha\) is strictly
completely monotone; in particular \(H_\alpha<0\) on \((0,\infty)\);
\item for every \(\alpha>0\), \(H_\alpha(x)=(1-\alpha)/x+O(1)\) as
\(x\downarrow0\).
\end{enumerate}
\end{theorem}
\begin{proof}
The series representation of the digamma function
\cite[Chapter~XII]{whittaker-watson} gives, for \(\alpha,x>0\),
\begin{equation}
\psi(x+\alpha)-\psi(x)
=\sum_{k=0}^{\infty}\left(\frac{1}{x+k}-\frac{1}{x+k+\alpha}\right)
=\alpha\sum_{k=0}^{\infty}\frac{1}{(x+k)(x+k+\alpha)} .
\end{equation}
The elementary telescoping identity
\begin{equation}
\frac1x=\sum_{k=0}^{\infty}\left(\frac{1}{x+k}-\frac{1}{x+k+1}\right)
=\sum_{k=0}^{\infty}\frac{1}{(x+k)(x+k+1)}
\end{equation}
expresses the subtracted term over the same index set. Therefore
\begin{equation}
\begin{aligned}
H_\alpha(x)
&=\alpha\sum_{k=0}^{\infty}
\frac{(x+k+1)-(x+k+\alpha)}
{(x+k)(x+k+1)(x+k+\alpha)}\\
&=\alpha(1-\alpha)\sum_{k=0}^{\infty}
\frac{1}{(x+k)(x+k+1)(x+k+\alpha)} .
\end{aligned}
\end{equation}
The terms are \(O((x+k)^{-3})\), so the series and its termwise derivatives
converge locally uniformly on \((0,\infty)\).

(i)--(ii) Each summand is a product of the completely monotone functions
\((x+c)^{-1}\), hence completely monotone, and strictly so; locally uniform
sums preserve the derivative-sign inequalities. The prefactor
\(\alpha(1-\alpha)\) is positive for \(0<\alpha<1\), zero for \(\alpha=1\)
(consistently with \(\psi(x+1)=\psi(x)+1/x\)), and negative for
\(\alpha>1\).

(iii) The \(k=0\) term is \(1/(x(x+1)(x+\alpha))=1/(\alpha x)+O(1)\), while
the sum over \(k\ge1\) converges to a finite limit as \(x\downarrow0\).
\end{proof}

\subsection{Sharp ranges for Simon's quotient and the reciprocal ratio}

Simon \cite{simon-moment} observed that logarithmic complete monotonicity of
\(1/F_\alpha\) alone does not imply the Bernstein property of \(F_\alpha\)
and left that question open. A complement representation proves sufficiency;
its witness is the measure \(\mu_\alpha\) below.

\begin{proposition}[Complement Laplace measure]\label{prop:complement}
Let \(0<\alpha<1\). Then
\begin{equation}\label{eq:complement}
1-F_\alpha(x)
=\frac{1}{\Gamma(\alpha)\Gamma(1-\alpha)}
\int_0^\infty e^{-xt}\,\mu_\alpha(dt),
\end{equation}
where \(\mu_\alpha=-dJ_\alpha\) is the positive finite measure obtained from
the decreasing function
\begin{equation}\label{eq:J-def}
J_\alpha(t)=\int_0^t(t-u)^{\alpha-1}(e^u-1)^{-\alpha}\,du
=\int_0^1(1-v)^{\alpha-1}\left(\frac{t}{e^{tv}-1}\right)^{\alpha}dv .
\end{equation}
In particular \(1-F_\alpha\in\CM\), \(F_\alpha'\in\CM\), and
\(F_\alpha\in\BF\). Since
\(\mu_\alpha((0,\infty))=\Gamma(\alpha)\Gamma(1-\alpha)\), the equivalent
L\'evy representation is
\begin{equation}\label{eq:simon-levy}
F_\alpha(x)
=\frac{1}{\Gamma(\alpha)\Gamma(1-\alpha)}
\int_{(0,\infty)}(1-e^{-xt})\,\mu_\alpha(dt).
\end{equation}
\end{proposition}
\begin{proof}
Put \(\beta=1-\alpha\). The beta integral
\cite[Eq.~5.12.1]{dlmf} and the substitution \(t=e^{-u}\) give
\begin{equation}
\frac{\Gamma(x+\alpha)}{\Gamma(x+1)}
=\frac{B(x+\alpha,\beta)}{\Gamma(\beta)}
=\frac{1}{\Gamma(\beta)}\int_0^\infty e^{-xu}\,(e^u-1)^{-\alpha}\,du ,
\end{equation}
hence, using \(F_\alpha(x)=x^{\beta}\,\Gamma(x+\alpha)/\Gamma(x+1)\) and
\(x^{-\alpha}=\Lap\{t^{\alpha-1}/\Gamma(\alpha)\}(x)\),
\begin{equation}
F_\alpha(x)
=\frac{x}{\Gamma(\alpha)\Gamma(1-\alpha)}
\int_0^\infty e^{-xt}J_\alpha(t)\,dt ,
\end{equation}
where \(J_\alpha\) is the Laplace convolution in \eqref{eq:J-def}; the second
form follows from \(u=tv\). For fixed \(v\in(0,1)\), the map
\(t\mapsto t/(e^{tv}-1)\) is positive and decreasing, so \(J_\alpha\) is
decreasing; the integrand is dominated by \((1-v)^{\alpha-1}v^{-\alpha}\),
integrable for \(0<\alpha<1\), so by dominated convergence
\begin{equation}
J_\alpha(0^+)=B(\alpha,1-\alpha)=\Gamma(\alpha)\Gamma(1-\alpha),
\qquad J_\alpha(\infty)=0 .
\end{equation}
Let \(\mu_\alpha=-dJ_\alpha\ge0\), a finite measure of total mass
\(\Gamma(\alpha)\Gamma(1-\alpha)\). Integration by parts gives
\begin{equation}
x\int_0^\infty e^{-xt}J_\alpha(t)\,dt
=J_\alpha(0^+)-\int_0^\infty e^{-xt}\,\mu_\alpha(dt),
\end{equation}
which is \eqref{eq:complement}; differentiating it gives
\(F_\alpha'=\Lap\{t\,\mu_\alpha\}/(\Gamma(\alpha)\Gamma(1-\alpha))\in\CM\),
and \(F_\alpha\ge0\) gives \(F_\alpha\in\BF\).
\end{proof}

\begin{theorem}[Sharp Bernstein range of Simon's quotient]
\label{thm:simon-sharp}
For \(\alpha\ge0\), \(F_\alpha\) is a Bernstein function on
\((0,\infty)\) if and only if \(0\le\alpha\le1\).
\end{theorem}
\begin{proof}
For \(\alpha\in\{0,1\}\), \(F_\alpha\equiv1\). For \(0<\alpha<1\),
Proposition~\ref{prop:complement} applies. For \(\alpha>1\),
Theorem~\ref{thm:identity}(ii) gives
\(F_\alpha'(x)=H_\alpha(x)F_\alpha(x)<0\) for all \(x>0\), so \(F_\alpha\)
is strictly decreasing. Every Bernstein function is nondecreasing because
\(F'\in\CM\).
\end{proof}

Simon already records the complete-Bernstein exclusion by an appeal to the
general theory \cite[p.~591]{simon-moment}. The next theorem supplies a direct
Pick-sign certificate on the cut.

\begin{theorem}[Boundary-phase certificate for Simon's quotient]
\label{thm:simon-cbf}
For every \(0<\alpha<1\), \(F_\alpha\in\BF\setminus\CBF\). Equivalently,
the L\'evy measure in \eqref{eq:simon-levy} is not absolutely continuous with
a completely monotone density.
\end{theorem}
\begin{proof}
The formula \(F_\alpha(z)=\Gamma(z+\alpha)/(z^\alpha\Gamma(z))\), with the
principal branch of \(z^\alpha\), extends \(F_\alpha\) holomorphically to
\(\mathbb C\setminus(-\infty,0]\) away from the poles \(z=-\alpha-k\) of
\(\Gamma(z+\alpha)\) on the cut (\(1/\Gamma\) is entire). If
\(F_\alpha\) were a complete Bernstein function, this extension would
coincide with the Pick extension, by uniqueness of analytic continuation
from \((0,\infty)\), and Proposition~\ref{prop:transport}(iii) would force
\(\Imag F_\alpha(z)\ge0\) throughout the open upper half-plane, hence also
for all boundary limits from above.

Fix \(t\in(\alpha,1)\), a nonempty interval, and let
\(z=-t+i\varepsilon\), \(\varepsilon\downarrow0\). The point \(-t\) is not
a pole (\(t\ne\alpha+k\)), and \(\Gamma\) is continuous there, so the limit
exists and equals
\begin{equation}
F_\alpha(-t+i0)
=\frac{\Gamma(\alpha-t)}{\Gamma(-t)}\;t^{-\alpha}\,e^{-i\pi\alpha},
\end{equation}
using \(\log(-t+i0)=\log t+i\pi\). Both arguments lie in \((-1,0)\), where
the Gamma function is strictly negative; hence
\(\Gamma(\alpha-t)/\Gamma(-t)>0\) and
\begin{equation}\label{eq:phase-witness}
\Imag F_\alpha(-t+i0)
=-\,\frac{\Gamma(\alpha-t)}{\Gamma(-t)}\;t^{-\alpha}\sin(\pi\alpha)<0 .
\end{equation}
This contradicts the Pick property, so \(F_\alpha\notin\CBF\). The final
statement follows from Definition~\ref{def:classes}(iv) and
\eqref{eq:simon-levy}.
\end{proof}

Koumandos and Pedersen study the reciprocal ratio below in generalized
Bernstein and generalized Stieltjes classes
\cite{koumandos-pedersen-two-parameter,
koumandos-pedersen-approximations}. The corollary records the ordinary
\(\LCM/\CM\) range delivered directly by the bridge.

\begin{corollary}[Sharp \(\LCM\) range of the reciprocal ratio]
\label{cor:reciprocal}
Let \(\alpha\ge0\) and
\(G_\alpha(x)=x^\alpha\Gamma(x)/\Gamma(x+\alpha)=1/F_\alpha(x)\). Then
\begin{equation}
G_\alpha\in\LCM
\quad\Longleftrightarrow\quad
G_\alpha\in\CM
\quad\Longleftrightarrow\quad
0\le\alpha\le1 .
\end{equation}
\end{corollary}
\begin{proof}
For \(0<\alpha<1\), \(-(\log G_\alpha)'=H_\alpha\in\CM\) by
Theorem~\ref{thm:identity}(i), so \(G_\alpha\in\LCM\subset\CM\); the
endpoints are trivial. For \(\alpha>1\),
\(G_\alpha'=-H_\alpha G_\alpha>0\) by Theorem~\ref{thm:identity}(ii), so
\(G_\alpha\) is strictly increasing and belongs to neither class.
\end{proof}

\subsection{Szab\'o's strict cutoff for the sign-adjusted bridge}

Szab\'o \cite{szabo-cm} proved that \(yH_d(y)\) is strictly completely
monotone for \(0<d<1\) and asked for the exact power cutoff in Open
Problem~1.5. The sign-adjusted formulation below also covers \(d>1\); the
case \(d=1\) is degenerate because \(H_1\equiv0\).

\begin{lemma}[Laplace witness for the bridge]\label{lem:kernel}
For \(d>0\) and \(y>0\),
\begin{equation}\label{eq:kernel}
H_d(y)=\int_0^\infty e^{-yt}\,g_d(t)\,dt,
\qquad
g_d(t)=\frac{1-e^{-dt}}{1-e^{-t}}-d ,
\end{equation}
with \(g_d(0^+)=0\) and \(g_d(\infty)=1-d\). Moreover
\begin{equation}\label{eq:kernel-derivative}
y\,H_d(y)=\int_0^\infty e^{-yt}\,g_d'(t)\,dt .
\end{equation}
\end{lemma}
\begin{proof}
The integral representation
\(\psi(y+d)-\psi(y)=\int_0^\infty e^{-yt}(1-e^{-dt})/(1-e^{-t})\,dt\)
\cite[Eq.~5.9.16]{dlmf} and \(d/y=\int_0^\infty e^{-yt}\,d\,dt\) give
\eqref{eq:kernel}. The function \(g_d\) is analytic on \([0,\infty)\) with
\(g_d(t)=\tfrac{d(1-d)}{2}\,t+O(t^2)\) at \(0\) and
\(g_d(t)=1-d+O(e^{-\min(1,d)t})\) at \(\infty\); in particular \(g_d\) is
bounded with bounded integrable derivative. Integration by parts,
\begin{equation}
y\int_0^\infty e^{-yt}g_d(t)\,dt
=\Bigl[-e^{-yt}g_d(t)\Bigr]_0^\infty+\int_0^\infty e^{-yt}g_d'(t)\,dt,
\end{equation}
and \(g_d(0^+)=0\) give \eqref{eq:kernel-derivative}.
\end{proof}

\begin{lemma}[Two-point convexity certificate]\label{lem:jensen}
For every \(t>0\), \(g_d'(t)\) has the strict sign of \(1-d\):
\(g_d'>0\) on \((0,\infty)\) for \(0<d<1\), and \(g_d'<0\) for \(d>1\).
\end{lemma}
\begin{proof}
Write \(\varphi_d(t)=(1-e^{-dt})/(1-e^{-t})\), so \(g_d'=\varphi_d'\) and
\begin{equation}
\varphi_d'(t)=\frac{N_d(t)}{(1-e^{-t})^2},
\qquad
N_d(t)=d\,e^{-dt}+(1-d)\,e^{-(d+1)t}-e^{-t},
\end{equation}
after expanding
\(N_d=de^{-dt}(1-e^{-t})-e^{-t}(1-e^{-dt})\). The function
\(s\mapsto e^{-st}\) is strictly convex on \(\R\) for each fixed \(t>0\).

For \(0<d<1\), the two-point probability measure
\(\nu=d\,\delta_{d}+(1-d)\,\delta_{d+1}\) has barycenter
\(d\cdot d+(1-d)(d+1)=1\). Strict convexity at this barycenter gives
\begin{equation}
d\,e^{-dt}+(1-d)\,e^{-(d+1)t}
=\int e^{-st}\,d\nu(s)>e^{-t\cdot1},
\end{equation}
that is, \(N_d(t)>0\).

For \(d>1\), the measure \(\nu=\delta_1+(d-1)\,\delta_{d+1}\) has total mass
\(d\) and normalized barycenter \((1+(d-1)(d+1))/d=d\). Strict convexity
applied to \(\nu/d\) gives
\begin{equation}
\frac1d\Bigl(e^{-t}+(d-1)\,e^{-(d+1)t}\Bigr)>e^{-dt},
\end{equation}
that is, \(N_d(t)<0\).
\end{proof}

\begin{theorem}[Strict cutoff at \(1\)]
\label{thm:szabo}
Let \(d\in(0,1)\cup(1,\infty)\) and define
\begin{equation}\label{eq:H-sharp}
H_d^\sharp(y)=\operatorname{sign}(1-d)\,H_d(y).
\end{equation}
Then \(H_d^\sharp(y)>0\) for \(y>0\), and
\(y^aH_d^\sharp(y)\) is strictly completely monotone on \((0,\infty)\)
exactly when \(a\le1\). Thus Szab\'o's cutoff is \(1\) for \(0<d<1\), and
the sign-adjusted bridge has the same cutoff for \(d>1\).
\end{theorem}
\begin{proof}
Theorem~\ref{thm:identity} gives \(H_d^\sharp>0\). For sufficiency,
Lemmas~\ref{lem:kernel} and~\ref{lem:jensen} give
\(yH_d^\sharp(y)=\Lap\{|g_d'|\}(y)\), where \(|g_d'(t)|>0\) for every
\(t>0\). Hence every alternating derivative of \(yH_d^\sharp\) is strictly
positive. For \(a<1\),
\(y^aH_d^\sharp=y^{-(1-a)}(yH_d^\sharp)\); the product rule shows strict
complete monotonicity.

For necessity, let \(a>1\). By Theorem~\ref{thm:identity}(iii),
\(H_d^\sharp(y)=|1-d|/y+O(1)\) as \(y\downarrow0\), so
\begin{equation}
y^aH_d^\sharp(y)=|1-d|\,y^{a-1}+O(y^{a})\longrightarrow0
\qquad(y\downarrow0).
\end{equation}
The function \(y^aH_d^\sharp\) is nonzero and nonnegative with boundary
value zero. Lemma~\ref{lem:one-point}(ii) rules out complete monotonicity.
\end{proof}

\section{Gamma quotients as polynomial witnesses}
\label{sec:quotient-obstructions}

Baricz \cite{baricz-turan} asked whether
\begin{equation}\label{eq:baricz-quotient}
x\longmapsto
\frac{\Gamma(x)\Gamma(x-a+b)}{\Gamma(x-a)\Gamma(x+b)}
\end{equation}
is a Bernstein function on \((a,\infty)\) for every \(a,b>0\). Substituting
\(y=x-a\) and using \(\Gamma(y+a)/\Gamma(y)\cdot\Gamma(y+b)/\Gamma(y+a+b)\),
the question concerns the symmetric family
\begin{equation}\label{eq:g-def}
g_{a,b}(y)=\frac{\Gamma(y+a)\,\Gamma(y+b)}{\Gamma(y)\,\Gamma(y+a+b)}
=g_{b,a}(y),\qquad y>0 .
\end{equation}
The reciprocal family is classical: from the digamma kernel
\cite[Eq.~5.9.16]{dlmf},
\begin{equation}\label{eq:baricz-kernel}
(\log g_{a,b})'(y)
=\int_0^\infty e^{-yt}\,
\frac{(1-e^{-at})(1-e^{-bt})}{1-e^{-t}}\,dt\;\in\CM,
\end{equation}
so \(g_{a,b}\) is strictly increasing to \(1\) and \(1/g_{a,b}\in\LCM\) for
\emph{all} \(a,b>0\) (cf.\ Bustoz and Ismail \cite{bustoz-ismail}). The
Bernstein question for \(g_{a,b}\) itself remains. For integer \(a\), its
witness is a Jacobi polynomial that decides the orthogonal strip and several
intermediate windows.

\begin{lemma}[Partial-fraction and Jacobi form of the witness]
\label{lem:baricz-witness}
Let \(a\in\N\) and let \(b>0\) with \(b\notin\{1,\dots,a-1\}\), so that
\begin{equation}
g_{a,b}(y)=\prod_{i=0}^{a-1}\frac{y+i}{y+b+i}
\end{equation}
has no common zero--pole pair. Then
\begin{equation}\label{eq:baricz-partial}
\begin{aligned}
1-g_{a,b}(y)&=\sum_{i=0}^{a-1}\frac{c_i}{y+b+i},\\
c_i&=(-1)^{a+i+1}\,
\frac{\prod_{j=0}^{a-1}(b+i-j)}{i!\,(a-1-i)!},\\
\sum_{i=0}^{a-1}c_i&=ab,
\end{aligned}
\end{equation}
and the inverse Laplace transform of \(g_{a,b}'\) is
\begin{equation}\label{eq:jacobi-witness}
\begin{aligned}
\rho_{a,b}(t)
&=t\,e^{-bt}\sum_{i=0}^{a-1}c_i\,e^{-it}\\
&=(-1)^{a+1}\,b\;t\,e^{-bt}\;
P^{(b-a,\,1)}_{a-1}\!\bigl(1-2e^{-t}\bigr),
\end{aligned}
\end{equation}
The density \(\rho_{a,b}(t)\,dt\) belongs to \(\mathcal M_L\). Here
\(P^{(\alpha,\beta)}_n\) is the Jacobi polynomial, defined for
arbitrary parameters by the terminating hypergeometric series
\cite[Eq.~18.5.7]{dlmf} and orthogonal on \((-1,1)\) exactly when
\(\alpha,\beta>-1\).
\end{lemma}
\begin{proof}
The gamma recurrence gives the product form. Since
\(1-g_{a,b}\) is a proper rational function with the simple poles
\(-(b+i)\), the partial fraction \eqref{eq:baricz-partial} holds with
\(c_i=-\operatorname*{Res}_{y=-(b+i)}g_{a,b}\); evaluating the residue,
\begin{equation}
c_i=-\,\frac{\prod_{j=0}^{a-1}(j-b-i)}{\prod_{j\ne i}(j-i)}
=(-1)^{a+i+1}\,\frac{\prod_{j=0}^{a-1}(b+i-j)}{i!\,(a-1-i)!},
\end{equation}
using \(\prod_{j\ne i}(j-i)=(-1)^ii!\,(a-1-i)!\). The mass identity follows
from \(y(1-g_{a,b}(y))\to ab\) as \(y\to\infty\). Differentiating
\eqref{eq:baricz-partial} gives
\(g_{a,b}'(y)=\sum_ic_i(y+b+i)^{-2}\), whose inverse Laplace transform is
the first expression in \eqref{eq:jacobi-witness}.

For the Jacobi form, set \(P(x)=\sum_{i=0}^{a-1}c_ix^i\). The term ratio
\begin{equation}
\frac{c_{i+1}}{c_i}
=-\,\frac{(a-1-i)\,(b+1+i)}{(i+1)\,(b+1-a+i)}
\end{equation}
identifies \(P(x)=c_0\cdot{}_2F_1(1-a,\,b+1;\,b+1-a;\,x)\), and the
classical formula
\(P^{(\alpha,\beta)}_n(1-2x)
=\binom{n+\alpha}{n}\,{}_2F_1(-n,\,1+\alpha+\beta+n;\,\alpha+1;\,x)\)
\cite[Eq.~18.5.7]{dlmf} with \(n=a-1\), \(\alpha=b-a\), \(\beta=1\) gives
\begin{equation}
P(x)=\frac{c_0}{\binom{b-1}{a-1}}\;P^{(b-a,1)}_{a-1}(1-2x)
=(-1)^{a+1}\,b\;P^{(b-a,1)}_{a-1}(1-2x),
\end{equation}
since
\(c_0=(-1)^{a+1}\Gamma(b+1)/(\Gamma(b-a+1)(a-1)!)\) and
\(\binom{b-1}{a-1}=\Gamma(b)/(\Gamma(a)\Gamma(b-a+1))\).
\end{proof}

\begin{theorem}[Orthogonal-strip classification for integer \(a\)]
\label{thm:baricz-classification}
Let \(a\in\N\) and \(b>a-1\) real. Then
\begin{equation}
g_{a,b}\in\BF
\quad\Longleftrightarrow\quad
a=1,
\end{equation}
and for \(a=1\), \(g_{1,b}(y)=y/(y+b)\) is a complete Bernstein function.
By the symmetry \(g_{a,b}=g_{b,a}\), the symmetric statement also holds when
\(b\in\N\) and \(a>b-1\). In particular Baricz's question has a negative
answer for every integer \(a\ge2\) with \(b>a-1\).
\end{theorem}
\begin{proof}
For \(a=1\), \(g_{1,b}(y)=y/(y+b)\) is rational with the single zero \(0\)
and pole \(-b\), a Pick function nonnegative on \((0,\infty)\), hence
complete Bernstein (Proposition~\ref{prop:transport}(iii)).

Let \(a\ge2\). The Jacobi parameters in \eqref{eq:jacobi-witness} satisfy
\(b-a>-1\) and \(1>-1\), so \(P^{(b-a,1)}_{a-1}\) is a genuine orthogonal
polynomial of degree \(a-1\ge1\) with respect to a positive weight on
\((-1,1)\); by the classical zero theorem \cite[Theorem~3.3.1]{szego}, all
its \(a-1\) zeros are simple and lie in the open interval \((-1,1)\). The
map \(t\mapsto1-2e^{-t}\) is an increasing bijection from \((0,\infty)\)
onto \((-1,1)\), so the continuous density \(\rho_{a,b}\) of
\eqref{eq:jacobi-witness} changes sign at least once (indeed exactly
\(a-1\) times) on \((0,\infty)\). By
Proposition~\ref{prop:witness}(i), \(\rho_{a,b}\) is the only candidate
representation of \(g_{a,b}'\), so \(g_{a,b}'\notin\CM\) and
\(g_{a,b}\notin\BF\).
\end{proof}

\begin{corollary}[Integral cancellation values]
\label{cor:baricz-cancellation}
Let \(a\in\N\) and \(m\in\{1,\dots,a-1\}\). Then
\(g_{a,m}=g_{m,a}\). The function is complete Bernstein for \(m=1\) and
fails to be Bernstein for \(m\ge2\).
\end{corollary}

\begin{proof}
Symmetry gives \(g_{a,m}=g_{m,a}\). Apply
Theorem~\ref{thm:baricz-classification} with first parameter \(m\) and second
parameter \(a>m-1\).
\end{proof}

\begin{proposition}[Interlacing regime]\label{prop:baricz-interlacing}
Let \(a\in\N\) and \(0<b\le1\). Then \(g_{a,b}\in\CBF\).
\end{proposition}
\begin{proof}
For \(b=1\) the product telescopes to \(g_{a,1}(y)=y/(y+a)\in\CBF\). For
\(0<b<1\), count signs in \eqref{eq:baricz-partial}: in
\(\prod_{j=0}^{a-1}(b+i-j)\), the factors with \(j\le i\) are positive and
the \(a-1-i\) factors with \(j\ge i+1\) are negative, so the product has
sign \((-1)^{a-1-i}\) and
\begin{equation}
\operatorname{sign}(c_i)=(-1)^{a+i+1}\,(-1)^{a-1-i}=+1 .
\end{equation}
Hence \(1-g_{a,b}(y)=\sum_ic_i/(y+b+i)\) with all \(c_i>0\), so for
\(\Imag y>0\),
\begin{equation}
\Imag g_{a,b}(y)
=\sum_{i=0}^{a-1}\frac{c_i\,\Imag y}{|y+b+i|^2}\ge0,
\end{equation}
and \(g_{a,b}\ge0\) on \((0,\infty)\); Proposition~\ref{prop:transport}(iii)
gives \(g_{a,b}\in\CBF\). Equivalently: the zeros \(0,-1,\dots\) and poles
\(-b,-1-b,\dots\) interlace exactly when \(b\le1\).
\end{proof}

\begin{proposition}[Odd intermediate windows are not Bernstein]
\label{prop:baricz-odd}
Let \(a\in\N\), \(a\ge3\), and let \(b\in(m,m+1)\) for an odd integer \(m\)
with \(1\le m\le a-2\). Then \(g_{a,b}\notin\BF\).
\end{proposition}
\begin{proof}
Since \(b\notin\{1,\dots,a-1\}\), Lemma~\ref{lem:baricz-witness} applies:
the unique candidate density of \(g_{a,b}'\) is
\(\rho_{a,b}(t)=te^{-bt}\sum_{i=0}^{a-1}c_ie^{-it}\). As \(t\to\infty\),
the sum converges to \(c_0\), so \(\rho_{a,b}\) has the sign of \(c_0\) on
a neighborhood of infinity. In the product \(\prod_{j=0}^{a-1}(b-j)\)
defining \(c_0\), exactly the \(a-1-m\) factors with \(j\ge m+1\) are
negative (as \(m<b<m+1\)), so
\begin{equation}
\operatorname{sign}(c_0)
=(-1)^{a+1}\,(-1)^{a-1-m}=(-1)^{m}=-1 .
\end{equation}
Hence \(\rho_{a,b}<0\) on an interval of positive length, and
Proposition~\ref{prop:witness}(i) gives \(g_{a,b}'\notin\CM\), that is,
\(g_{a,b}\notin\BF\).
\end{proof}

\begin{remark}[Geometry of the witness]
\label{rem:baricz}
For \((a,b)=(2,3)\),
\(P^{(1,1)}_{1}(1-2x)=2-4x\), so
\(\rho_{2,3}(t)=te^{-3t}(-6+12e^{-t})\), negative exactly for \(t>\log2\);
Appendix~\ref{app:certificates} gives an independent rational derivative
certificate. In general, the argument of the damped Jacobi witness sweeps
the full orthogonality interval, reducing failure for \(a\ge2\) to the zero
theorem for positive-degree orthogonal polynomials. The remaining range is
\(1<b<a-1\), where quasi-orthogonality replaces orthogonality.
Proposition~\ref{prop:baricz-odd} and
Corollary~\ref{cor:baricz-cancellation} complete the classification for
\(a\in\{2,3\}\). The first unresolved even window is \(a=4\),
\(b\in(2,3)\); see Problem~\ref{prob:baricz-real}.
\end{remark}

\begin{counterexample}[Dulac--Simon cumulative-Tsallis ratio is not
completely monotone]\label{cx:dulac-simon}
Let
\begin{equation}
\Phi(s)=\frac{(s+1)^2}{2s^2(2s+1)}
\left(1-\frac{\Gamma(s+1)^2}{\Gamma(2s+1)}\right).
\end{equation}
Then \(\Phi\) is not completely monotone on \((0,\infty)\); a fortiori it is
not completely monotone on the source domain \((-1/2,\infty)\) of Dulac and
Simon \cite[Remark~11(c)]{dulac-simon}.
\end{counterexample}
\begin{proof}
For \(s>0\), \(B(s+1,s+1)=\int_0^1[x(1-x)]^s\,dx\). Substituting
\(t=-\log(x(1-x))\), whose two branches are \(x=(1\pm y(t))/2\) with
\(y(t)=\sqrt{1-4e^{-t}}\), gives
\begin{equation}\label{eq:beta-kernel}
B(s+1,s+1)=\int_{\log4}^{\infty}e^{-st}w(t)\,dt,
\qquad
w(t)=\frac{2e^{-t}}{\sqrt{1-4e^{-t}}}=y'(t) .
\end{equation}
Using \(\Gamma(s+1)^2/\Gamma(2s+1)=(2s+1)B(s+1,s+1)\) and the partial
fractions
\begin{equation}
\frac{(s+1)^2}{2s^2(2s+1)}
=\frac{1}{2s^2}+\frac{1}{2(2s+1)},
\qquad
\frac{(s+1)^2}{2s^2}=\frac12+\frac1s+\frac{1}{2s^2},
\end{equation}
we obtain \(\Phi=\Lap\{K\}\) with, for \(a=\log4\),
\begin{equation}
K(t)=\frac t2+\frac14e^{-t/2}
-\mathbf 1_{\{t>a\}}
\left[\frac12w(t)+\int_a^tw(u)\,du+\frac12\int_a^t(t-u)w(u)\,du\right].
\end{equation}
Since \(y'=w\), \(y(a)=0\), and
\(\frac{d}{du}[\log\frac{1+y}{1-y}-2y]=2y'\frac{y^2}{1-y^2}=y\), the
integrals evaluate in closed form: for \(t>a\),
\begin{equation}\label{eq:K-closed}
K(t)=\frac t2+\frac14e^{-t/2}
-\frac{e^{-t}}{\sqrt{1-4e^{-t}}}
-\frac12\log\frac{1+\sqrt{1-4e^{-t}}}{1-\sqrt{1-4e^{-t}}} .
\end{equation}
At \(t=a+\varepsilon\), \(1-4e^{-t}=1-e^{-\varepsilon}\sim\varepsilon\), so
the third term is \(\sim1/(4\sqrt\varepsilon)\) while the logarithm is
\(O(\sqrt\varepsilon)\) and the first two terms stay bounded. Hence
\begin{equation}
K(\log4+\varepsilon)\sim-\frac{1}{4\sqrt{\varepsilon}}
\qquad(\varepsilon\downarrow0),
\end{equation}
and \(K<0\) on a right-neighborhood of \(\log4\). The density belongs to the
admissible signed class of Proposition~\ref{prop:witness}(i):
\begin{equation}\label{eq:K-admissibility}
K(t)=O\!\left((t-\log4)^{-1/2}\right)
\quad(t\downarrow\log4),
\qquad
K(t)=O(e^{-t/2})
\quad(t\to\infty).
\end{equation}
Thus \(\int_0^\infty e^{-st}|K(t)|\,dt<\infty\) for every \(s>0\). By
Proposition~\ref{prop:witness}(i), \(K\) is the only candidate
representation of \(\Phi\), so \(\Phi\notin\CM\).
\end{proof}

\section{Boundary geometry for a fractional symbol}
\label{sec:symbols}

Wakrim \cite{wakrim-w} introduces a Volterra-type fractional time operator
whose symbol, in the normalization of the source, is
\begin{equation}
\Phi_{\alpha,\beta}(s)
=s^{\alpha}\left(1+(1-\alpha)s^{\alpha-1}\right)^{-\beta},
\qquad 0<\alpha<1,\ \beta\ge0.
\end{equation}
Version~2 of the source proves the sharp Bernstein threshold
\(\Phi_{\alpha,\beta}\in\BF\) iff \(0\le\beta\le1\). The result below gives
the complete-Bernstein upgrade: the same interval is the exact \(\CBF\)
range.

\begin{theorem}[The Bernstein and complete Bernstein regimes coincide]
\label{thm:wakrim}
For every \(0<\alpha<1\) and \(\beta\ge0\), the following are equivalent:
\begin{enumerate}[label=\textup{(\roman*)},leftmargin=2.3em]
\item \(0\le\beta\le1\);
\item \(\Phi_{\alpha,\beta}\in\BF\);
\item \(\Phi_{\alpha,\beta}\in\CBF\).
\end{enumerate}
\end{theorem}
\begin{proof}
Set \(a=1-\alpha\in(0,1)\) and \(p=\alpha+(1-\alpha)\beta\), so that on
\((0,\infty)\)
\begin{equation}
\Phi_{\alpha,\beta}(s)
=s^{1-a}\left(1+as^{-a}\right)^{-\beta}
=\frac{s^{\,p}}{(s^a+a)^{\beta}},
\qquad p-\beta a=\alpha .
\end{equation}
(iii)\(\Rightarrow\)(ii) is the inclusion \(\CBF\subset\BF\), and
(ii)\(\Rightarrow\)(i) is the source threshold:
\(\Phi_{\alpha,\beta}\notin\BF\) for \(\beta>1\)
\cite[Corollary~3.8]{wakrim-w}.

(i)\(\Rightarrow\)(iii). For \(\beta=0\),
\(\Phi_{\alpha,0}(s)=s^\alpha\in\CBF\). Let \(0<\beta\le1\), so
\(0<\alpha<p\le1\). The formula \(s^p/(s^a+a)^\beta\), with principal
powers, extends \(\Phi_{\alpha,\beta}\) holomorphically to
\(\mathbb C\setminus(-\infty,0]\): for \(\Imag z>0\) write
\(\theta=\arg z\in(0,\pi)\); then \(\arg(z^a)=a\theta\in(0,\pi)\), so
\(\Imag(z^a)>0\) and \(z^a+a\ne0\), and adding the positive constant \(a\)
increases the real part while preserving the imaginary part, whence
\begin{equation}
0<\arg\bigl(z^a+a\bigr)<a\theta .
\end{equation}
Therefore
\begin{equation}\label{eq:wakrim-arg}
\arg\Phi_{\alpha,\beta}(z)
=p\,\theta-\beta\arg\bigl(z^a+a\bigr)
\in\bigl((p-\beta a)\,\theta,\;p\,\theta\bigr)
=(\alpha\theta,\;p\theta)\subset(0,\pi),
\end{equation}
so \(\Imag\Phi_{\alpha,\beta}(z)>0\): the symbol maps the upper half-plane
to itself. Since \(\Phi_{\alpha,\beta}>0\) on \((0,\infty)\),
Proposition~\ref{prop:transport}(iii) gives
\(\Phi_{\alpha,\beta}\in\CBF\).
\end{proof}

\section{Stieltjes witnesses: recognition and separation}
\label{sec:stieltjes}

\subsection{A positive sine-mixture witness}

Mishra and Swaminathan \cite{mishra-swaminathan} study the Ramanujan integral
\begin{equation}\label{eq:ramanujan-integral}
I_R(x)=\int_0^\infty\frac{e^{-xt}}{t\,(\pi^2+\log^2t)}\,dt,
\qquad x>0,
\end{equation}
whose density occurs in Ramanujan's formula for the \(\nu\)-function
\cite{hardy}. Their Theorem~8 proves that a Bernstein primitive exists. Its
displayed antiderivative kernel omits one factor \(t^{-1}\), as direct
differentiation shows. The correctly normalized primitive is
\begin{equation}\label{eq:ramanujan-primitive}
P_R(x)=c+\int_0^\infty
\frac{1-e^{-xt}}{t^2(\pi^2+\log^2t)}\,dt,
\qquad c\ge0,
\end{equation}
and satisfies \(P_R'=I_R\).

\begin{theorem}[Stieltjes recognition of the Ramanujan integral]
\label{thm:ramanujan}
The density
\begin{equation}\label{eq:ramanujan-phi}
\phi_0(t)=\frac{1}{t(\pi^2+\log^2t)}
\end{equation}
is completely monotone on \((0,\infty)\). Consequently \(I_R\) is a
Stieltjes function, and the primitive \(P_R\) in
\eqref{eq:ramanujan-primitive} is a complete Bernstein function.
\end{theorem}
\begin{proof}
For \(u\in\R\), elementary integration gives
\(\int_0^1e^{-au}\sin(\pi a)\,da=\pi(1+e^{-u})/(u^2+\pi^2)\). Substituting
\(u=\log t\) yields the sine-mixture identity
\begin{equation}\label{eq:sine-mixture}
\phi_0(t)=\frac{1}{\pi\,(1+t)}\int_0^1t^{-a}\sin(\pi a)\,da .
\end{equation}
Each \(t^{-a}\) with \(0<a<1\) is completely monotone, \(\sin(\pi a)>0\) on
\((0,1)\), and \((1+t)^{-1}\in\CM\); positive mixtures and products of
completely monotone functions are completely monotone
(Proposition~\ref{prop:transport}(i)), so \(\phi_0\in\CM\).

By Proposition~\ref{prop:transport}(ii), \(I_R=\Lap\{\phi_0\}\in\St\).
Moreover \(\phi_0(t)/t\in\CM\), because products of completely monotone
functions are completely monotone, and
\(\int_0^\infty(1\wedge t)\phi_0(t)t^{-1}\,dt<\infty\): near \(0\) the
integrand after multiplication by \(t\) is
\(1/[t(\pi^2+\log^2t)]\), and near infinity it is
\(1/[t^2(\pi^2+\log^2t)]\). Hence \(P_R\) is a Bernstein function with
completely monotone L\'evy density \(\phi_0(t)/t\), so
Definition~\ref{def:classes}(iv) gives \(P_R\in\CBF\).
\end{proof}

\subsection{Failure of product closure}

The Stieltjes cone is convex and fails product closure. The elementary pair
\(1/x\in\St\), \(1/x^2\notin\St\) uses a double pole at the origin. The
function below, arising in the study of windings of planar processes and
Asian options \cite{jedidi-vakeroudis}, gives a nondegenerate pair attached
to one complete Bernstein function. A Pick-sign certificate decides the
product.

\begin{proposition}[Stieltjes structure]
\label{prop:arsinh-structure}
Let \(a(x)=\arsinh\sqrt x=\log(\sqrt x+\sqrt{1+x})\) on \((0,\infty)\).
Then:
\begin{enumerate}[label=\textup{(\roman*)},leftmargin=2.3em]
\item \(a\in\CBF\);
\item \(a'\in\St\), with arcsine representation
\begin{equation}\label{eq:arcsine}
a'(x)=\frac{1}{2\sqrt{x}\sqrt{1+x}}
=\frac12\int_0^1\frac{1}{x+t}\cdot
\frac{dt}{\pi\sqrt{t(1-t)}} ;
\end{equation}
\item \(1/a\in\St\).
\end{enumerate}
\end{proposition}
\begin{proof}
(i) The function \(a\) is positive on \((0,\infty)\) and extends
holomorphically to \(\mathbb C\setminus(-\infty,0]\): for such \(z\), the
principal roots satisfy \(1+z\notin(-\infty,0]\), and for \(\Imag z>0\) both
\(\sqrt z\) and \(\sqrt{1+z}\) lie in the open first quadrant, so their sum
does as well; in particular the sum avoids \((-\infty,0]\) and
\(\Imag a(z)=\arg\bigl(\sqrt z+\sqrt{1+z}\bigr)\in(0,\pi/2)\). Thus \(a\) is
nonnegative on \((0,\infty)\) and maps the upper half-plane to itself, and
Proposition~\ref{prop:transport}(iii) gives \(a\in\CBF\).

(ii) The first equality is a computation. For the second, put
\(t=\sin^2\theta\) and then \(u=\tan\theta\):
\begin{equation}
\int_0^1\frac{dt}{(x+t)\sqrt{t(1-t)}}
=2\int_0^{\pi/2}\frac{d\theta}{x+\sin^2\theta}
=\frac{\pi}{\sqrt{x(1+x)}}.
\end{equation}
This is the stated arcsine representation.

(iii) By (i) and the bijection
\(\St\setminus\{0\}\ni f\mapsto1/f\in\CBF\setminus\{0\}\)
(Proposition~\ref{prop:transport}(iii)), \(a\in\CBF\) gives \(1/a\in\St\).
\end{proof}

\begin{counterexample}[Arsinh-square logarithmic derivative]
\label{cx:arsinh}
Let \(\varphi(x)=\arsinh^2\sqrt x\) and
\begin{equation}
h(x)=\frac{\varphi'(x)}{\varphi(x)}
=2\,a'(x)\cdot\frac{1}{a(x)}
=\frac{1}{\sqrt x\,\sqrt{1+x}\,\arsinh\sqrt x}.
\end{equation}
Then \(h\) is a product of the two Stieltjes functions of
Proposition~\ref{prop:arsinh-structure}. It is completely monotone and lies
outside the Stieltjes class.
\end{counterexample}
\begin{proof}
Complete monotonicity follows from \(\St\subset\CM\) and product closure
(Proposition~\ref{prop:transport}). For the exclusion we compute a Pick-sign
certificate from Proposition~\ref{prop:witness}(iii). Evaluate the
holomorphic extension of \(h\) at \(z=-t+i0\), \(t>1\). On principal
branches, \(\sqrt z=i\sqrt t\), \(\sqrt{1+z}=i\sqrt{t-1}\), and
\begin{equation}
\arsinh(i\sqrt t)
=\log\bigl(i(\sqrt t+\sqrt{t-1})\bigr)
=A(t)+\frac{\pi i}{2},
\qquad
A(t)=\arcosh\sqrt t>0 .
\end{equation}
Therefore
\begin{equation}
h(-t+i0)
=\frac{1}{i\sqrt t\cdot i\sqrt{t-1}\cdot\bigl(A(t)+\frac{\pi i}2\bigr)}
=-\frac{A(t)-\frac{\pi i}{2}}
{\sqrt{t(t-1)}\,\bigl(A(t)^2+\frac{\pi^2}{4}\bigr)},
\end{equation}
so
\begin{equation}\label{eq:arsinh-witness}
\Imag h(-t+i0)
=\frac{\pi/2}{\sqrt{t(t-1)}\,\bigl(A(t)^2+\frac{\pi^2}{4}\bigr)}>0
\qquad(t>1).
\end{equation}
A Stieltjes function satisfies \(\Imag h\le0\) on the upper half-plane
(Proposition~\ref{prop:witness}(iii)). Equation~\eqref{eq:arsinh-witness}
has the opposite sign, hence \(h\notin\St\).
\end{proof}

\section{Hausdorff witnesses and the discrete boundary}
\label{sec:discrete}

Hausdorff witnesses enter through Bernstein interpolation
(Theorem~\ref{thm:grws}), a derived sequence
(Counterexample~\ref{cx:berg-pedersen}), and coefficient sequences
(Theorem~\ref{thm:lp-rigidity}).

\subsection{Bernstein interpolation of weighted-shift sequences:
a complete classification}

Benhida, Curto, and Exner \cite{benhida-curto-exner} attach to a
geometrically regular weighted shift the weights-squared sequence
\begin{equation}
a_n=\frac{p^n+N}{p^n+D},\qquad p>1,\ -1<N,D<1,
\end{equation}
and ask, in their Sector~II (\(-1<N<0<D\le-N\)), whether \((a_n)\) is the
restriction to \(\N_0\) of a Bernstein function. The answer below is a
complete classification of the whole parameter square; the Sector~II
question is the case \textup{(iii)}.

\begin{theorem}[Bernstein interpolability of GRWS sequences]
\label{thm:grws}
Interpolation at \(n=0\) is understood through \(F(0)=F(0^+)\). Let
\(p>1\), \(q=p^{-1}\), and \(-1<N,D<1\). The sequence
\(a_n=(p^n+N)/(p^n+D)\) is interpolated by a Bernstein function if and only if
\begin{equation}
N=D
\qquad\text{or}\qquad
\bigl(N<D\ \text{and}\ D\le0\bigr).
\end{equation}
Precisely:
\begin{enumerate}[label=\textup{(\roman*)},leftmargin=2.3em]
\item if \(N=D\), then \(a_n\equiv1\) and \(F\equiv1\) interpolates;
\item if \(N<D\le0\), then
\begin{equation}\label{eq:grws-interpolant}
F(x)=\frac{1+Nq^{x}}{1+Dq^{x}},\qquad x>0,
\end{equation}
is a Bernstein interpolant: with \(c=\log p\) and \(u=e^{-cx}\),
\begin{equation}\label{eq:grws-derivative}
F'(x)=c\,(D-N)\,\frac{u}{(1+Du)^2}
=c\,(D-N)\sum_{m\ge1}m\,(-D)^{m-1}e^{-mcx},
\end{equation}
a convergent series with nonnegative coefficients, so \(F'\in\CM\);
\item if \(N<D\) and \(0<D<1\), no Bernstein function interpolates
\((a_n)\): the increments \(d_n=a_{n+1}-a_n\) have the unique signed
Hausdorff representation
\begin{equation}\label{eq:grws-atoms}
d_n=(D-N)\sum_{m\ge1}(1-q^m)(-D)^{m-1}q^{mn},
\end{equation}
whose atom at \(q^2\) has mass \(-(D-N)D(1-q^2)<0\);
\item if \(N>D\), the sequence is strictly decreasing and therefore has no
Bernstein interpolant.
\end{enumerate}
\end{theorem}
\begin{proof}
(i) and (iv) are immediate (a Bernstein function has \(F'\in\CM\ge0\)).

(ii) \(F\) is well defined and positive: \(1+Nq^x>0\) and \(1+Dq^x>0\)
since \(|N|,|D|<1\) and \(0<q^x<1\); and \(F(n)=a_n\) by construction.
Differentiating \(F=1+(N-D)u/(1+Du)\) in \(x\) with \(u=e^{-cx}\) gives the
closed form in \eqref{eq:grws-derivative}, and expanding
\((1+Du)^{-2}=\sum_{m\ge1}m(-D)^{m-1}u^{m-1}\)---absolutely convergent since
\(|Du|\le|D|<1\)---gives the series. For \(D\le0\) every coefficient
\(m(-D)^{m-1}\) is nonnegative, each \(e^{-mcx}\) is completely monotone,
and the convergence is locally uniform, so \(F'\in\CM\); with \(F\ge0\) this
gives \(F\in\BF\).

(iii) The first differences telescope to
\(d_n=(D-N)(1-q)q^n/[(1+Dq^n)(1+Dq^{n+1})]\). With \(u=q^n\), the partial
fraction
\begin{equation}
\frac{(1-q)u}{(1+Du)(1+Dqu)}
=\frac1D\left(\frac{1}{1+Dqu}-\frac{1}{1+Du}\right)
\end{equation}
and absolutely convergent geometric expansions give
\eqref{eq:grws-atoms}, i.e.\ the finite signed atomic representing measure
\(\sigma=(D-N)\sum_{m\ge1}(1-q^m)(-D)^{m-1}\delta_{q^m}\) with
\(\|\sigma\|\le(D-N)/(1-D)\). Its atom at \(q^2\) is negative. If a
Bernstein interpolant existed, Lemma~\ref{lem:bf-increments} would make
\((d_n)\) a positive Hausdorff moment sequence; by the determinacy clause of
Proposition~\ref{prop:witness}(ii), the positive measure would coincide with
\(\sigma\), contradicting the negative atom.
\end{proof}

On the range \(N<D\le0\), equation~\eqref{eq:grws-derivative} identifies the
nonconstant part of the unique L\'evy measure of the displayed interpolant:
\begin{equation}\label{eq:grws-levy}
(D-N)\sum_{m\ge1}(-D)^{m-1}\delta_{m\log p}.
\end{equation}
It is nonzero and purely atomic, so this interpolant lies outside \(\CBF\).
The existence of another, complete Bernstein interpolant is
Problem~\ref{prob:grws-cbf}.

\subsection{The Berg--Pedersen derived sequence is not Hausdorff}

Berg and Pedersen \cite{berg-pedersen-horn} attach to their family of
Horn--Bernstein functions a derived sequence \(a_0=1\),
\(a_n=1/t_n-1/t_{n-1}\), where \((t_n)\) is built from the source recurrence
(reproduced in Appendix~\ref{app:certificates}), and ask whether \((a_n)\)
is a Hausdorff moment sequence.

\begin{counterexample}[Seventh-difference violation]\label{cx:berg-pedersen}
The Berg--Pedersen derived sequence is not a Hausdorff moment sequence on
\([0,1]\):
\begin{equation}
H_7(0)=\sum_{j=0}^{7}(-1)^j\binom7j a_j
=-\frac{28629387882812}{1395498975394445}<0 .
\end{equation}
\end{counterexample}
\begin{proof}
The values \(a_0,\dots,a_7\) are explicit rationals computed exactly from
the source recurrence in Appendix~\ref{app:certificates}. The displayed
seventh difference is a finite exact rational computation, and the witness
criterion \eqref{eq:hausdorff-witness} of
Proposition~\ref{prop:witness}(ii) converts its negativity into the
obstruction.
\end{proof}

\subsection{Laguerre--P\'olya rigidity: only pure exponentials have
completely monotone coefficients}

Recall that the Laguerre--P\'olya class \(\LP\) consists of the real entire
functions that are locally uniform limits of real polynomials with only
real zeros; equivalently, of the functions
\(Cx^me^{\sigma x-\beta x^2}\prod_k(1+x/x_k)e^{-x/x_k}\) with
\(C,\sigma\in\R\), \(\beta\ge0\), \(m\in\N_0\), \(x_k\in\R\setminus\{0\}\),
\(\sum_kx_k^{-2}<\infty\) \cite{levin}. Wang and Yang \cite{wang-yang}
prove complete-monotonicity results for coefficient sequences attached to
the Riemann \(\Xi\)-function and the partition function, and ask (open
problem following their Conjectures~1.10--1.11) whether every
\(\psi(x)=\sum_{n\ge0}\gamma_nx^n/n!\in\LP\) with positive, decreasing
coefficient sequence has \((\gamma_n)\) completely monotone. The following
rigidity theorem answers this in a strong form.

Wang and Yang use the forward difference
\(\Delta_{\!f}a_n=a_{n+1}-a_n\) and test
\((-1)^r\Delta_{\!f}^{,r}a_n\ge0\). Our convention
\(\Delta a_n=a_n-a_{n+1}\) gives the same inequalities as
\(\Delta^ra_n\ge0\).

\begin{theorem}[LP rigidity]\label{thm:lp-rigidity}
Let \(\psi(x)=\sum_{n\ge0}\gamma_nx^n/n!\) be a Laguerre--P\'olya function,
not identically zero. Then \((\gamma_n)_{n\ge0}\) is completely monotone if
and only if
\begin{equation}
\psi(x)=\gamma_0\,e^{qx}
\qquad\text{for some }q\in[0,1]\ \text{and}\ \gamma_0>0 .
\end{equation}
\end{theorem}
\begin{proof}
If \(\psi=\gamma_0e^{qx}\) with \(q\in[0,1]\), then
\(\gamma_n=\gamma_0q^n\), a Hausdorff moment sequence (of
\(\gamma_0\delta_q\)), hence completely monotone.

Conversely, let \((\gamma_n)\) be completely monotone. If \(\gamma_0=0\),
then \(0\le\gamma_n\le\gamma_0\) forces \(\gamma\equiv0\) and \(\psi\equiv0\),
excluded; so \(\gamma_0>0\). By Proposition~\ref{prop:witness}(ii) there is
a finite positive measure \(\rho\ne0\) on \([0,1]\) with
\(\gamma_n=\int_0^1s^n\,d\rho(s)\). Absolute convergence, bounded by
\(e^{|x|}\rho([0,1])\), permits interchange of the sum and integral:
\begin{equation}\label{eq:egf-witness}
\psi(x)=\sum_{n\ge0}\frac{x^n}{n!}\int_0^1s^n\,d\rho(s)
=\int_0^1e^{sx}\,d\rho(s),
\qquad x\in\R .
\end{equation}
The integrand is strictly positive for every real \(x\), so \(\psi>0\) on
\(\R\): \(\psi\) has no real zeros. A Laguerre--P\'olya function without
real zeros has the form \(\psi(x)=\gamma_0e^{\sigma x-\beta x^2}\) with
\(\beta\ge0\) \cite{levin}. If \(\beta>0\), then \(\psi(x)\to0\) as
\(x\to+\infty\). Equation~\eqref{eq:egf-witness} also gives
\(\psi(x)\ge\rho([0,1])=\gamma_0>0\) for all \(x\ge0\) (since
\(e^{sx}\ge1\) there), a contradiction. Hence \(\beta=0\) and
\(\gamma_n=\gamma_0\sigma^n\). Complete monotonicity forces
\(\gamma_1=\gamma_0\sigma\ge0\) and
\(\Delta\gamma_0=\gamma_0(1-\sigma)\ge0\), so \(\sigma\in[0,1]\).
\end{proof}

\begin{remark}[Role of real-rootedness]
The Laguerre--P\'olya hypothesis supplies the rigidity. General positive
exponential mixtures may have nonreal zeros: for example,
\((e^x-1)/x=\int_0^1e^{sx}\,ds\) has completely monotone coefficients
\(1/(n+1)\) and zeros \(2\pi ik\), \(k\ne0\). Within \(\LP\), positivity of
the witness measure collapses the zero set.
\end{remark}

\noindent
The next proposition quantifies the failure for one-factor
Laguerre--P\'olya functions.

\begin{proposition}[Exact difference table for one-factor atoms]
\label{prop:atoms}
Let \(0<q<1\), \(c\ge0\), and \(\gamma_n=(1+cn)q^n\), the coefficient
sequence of the Laguerre--P\'olya function \((1+cqx)e^{qx}\). Then for all
\(k\ge1\) and \(n\ge0\),
\begin{equation}\label{eq:atom-differences}
\Delta^k\gamma_n
=q^n(1-q)^{k-1}\bigl[(1-q)(1+cn)-ckq\bigr].
\end{equation}
Consequently \((\gamma_n)\) is positive; it is nonincreasing if and only if
\(c\le(1-q)/q\); and for any \(c>0\), taking \(n=0\) and
\(k>(1-q)/(cq)\) makes \eqref{eq:atom-differences} negative, in accordance
with Theorem~\ref{thm:lp-rigidity}.
\end{proposition}
\begin{proof}
By linearity it suffices to difference the two pieces. First,
\(\Delta^kq^n=(1-q)^kq^n\). Second, we claim
\(\Delta^k(nq^n)=q^n(1-q)^{k-1}\bigl(n(1-q)-kq\bigr)\) for \(k\ge1\): the
case \(k=1\) is \(nq^n-(n+1)q^{n+1}=q^n(n(1-q)-q)\), and if the formula
holds for \(k\), then
\begin{equation}
\begin{aligned}
\Delta^{k+1}(nq^n)
&=q^n(1-q)^{k-1}\bigl(n(1-q)-kq\bigr)
\\
&\quad-q^{n+1}(1-q)^{k-1}
\bigl((n+1)(1-q)-kq\bigr)\\
&=q^n(1-q)^{k-1}
\Bigl[n(1-q)^2-q(1-q)-kq(1-q)\Bigr]\\
&=q^n(1-q)^{k}\bigl(n(1-q)-(k+1)q\bigr).
\end{aligned}
\end{equation}
Adding \(c\) times the second formula to the first gives
\eqref{eq:atom-differences}. Monotonicity reduces to the worst case
\(n=0\), \(k=1\), i.e.\ \((1-q)\ge cq\).
\end{proof}

\begin{counterexample}[The Wang--Yang problem has a negative answer]
\label{cx:wang-yang}
The Laguerre--P\'olya function
\begin{equation}
\begin{aligned}
\psi(x)=\Bigl(1+\frac{4x}{15}\Bigr)e^{2x/3}
&=e^{14x/15}\Bigl(1+\frac{x}{15/4}\Bigr)\\
&\quad \times e^{-x/(15/4)}
\end{aligned}
\end{equation}
has coefficients \(\gamma_n=(1+\tfrac{2n}{5})(\tfrac23)^n\), positive and
strictly decreasing since \(c=\tfrac25\le(1-q)/q=\tfrac12\), yet
\begin{equation}
\Delta^2\gamma_0
=1-2\cdot\frac{14}{15}+\frac45
=-\frac{1}{15}<0 .
\end{equation}
\end{counterexample}
\begin{proof}
Instantiate Proposition~\ref{prop:atoms} at \(q=2/3\), \(c=2/5\):
\begin{equation}
\Delta^k\gamma_n
=\Bigl(\frac23\Bigr)^n\Bigl(\frac13\Bigr)^{k-1}
\left[\frac13\left(1+\frac{2n}{5}\right)-\frac{4k}{15}\right].
\end{equation}
At \((n,k)=(0,2)\), this equals
\begin{equation}
\frac13\left[\frac13-\frac{8}{15}\right]
=-\frac{1}{15}.
\end{equation}
\end{proof}

\section{Conclusion and open problems}\label{sec:concluding}

The paper uses one proof architecture: normalize a class question, compute or
transport the unique admissible witness, and decide its sign. The abstract
determinacy step is elementary. Its value is realized when the resulting sign
problem has a recognizable structure.

Three witness modalities account for the retained applications. Laplace
witnesses produce the digamma triple product, Simon's complement measure, and
the Jacobi-polynomial density for integer gamma quotients. Stieltjes witnesses
and Pick-sign certificates give the Ramanujan recognition theorem, the Wakrim
upgrade, and the arsinh separation. Hausdorff witnesses convert Bernstein
increments and coefficient sequences into signed atoms or finite differences.
These reductions replace derivative towers by classical polynomial geometry,
two-point convexity, boundary arguments, or exact arithmetic.

The method has a clear limitation. It simplifies a problem only when the
inverse datum can be computed and signed more readily than the initial class
condition. The two questions below identify where the present witnesses are
explicit and their remaining sign geometry is unresolved.

\begin{problem}[The Bernstein region of the Baricz family]
\label{prob:baricz-real}
Determine the exact region
\begin{equation}\label{eq:baricz-open-region}
\{(a,b)\in(0,\infty)^2:g_{a,b}\in\BF\}
\end{equation}
for the family \eqref{eq:g-def}. Theorem~\ref{thm:baricz-classification},
Corollary~\ref{cor:baricz-cancellation}, and
Propositions~\ref{prop:baricz-interlacing} and~\ref{prop:baricz-odd} decide
the orthogonal strip, the integral cancellation values, the interlacing
range, and the odd intermediate windows. The first unresolved even window is
integer \(a=4\), \(b\in(2,3)\); nonintegral parameters also remain open. In
these regimes the task is to control the interior sign of a
quasi-orthogonal or nonpolynomial Laplace witness
\cite{driver-jordaan-qoj}.
\end{problem}

\begin{problem}[Complete Bernstein interpolation of GRWS sequences]
\label{prob:grws-cbf}
On the affirmative range \(N<D\le0\) of Theorem~\ref{thm:grws}, the explicit
interpolant \eqref{eq:grws-interpolant} has the atomic L\'evy measure
\eqref{eq:grws-levy} and lies outside \(\CBF\). Decide whether the same
sequence \((p^n+N)/(p^n+D)\) admits another interpolant in \(\CBF\). This is
the Pick-class version of the interpolation problem arising from the
weighted shifts of \cite{benhida-curto-exner}.
\end{problem}

\appendix
\renewcommand{\thetheorem}{\Alph{section}.\arabic{theorem}}

\section{Exact rational certificates}\label{app:certificates}

This appendix records the finite exact computations invoked in the text. All
statements are equalities or inequalities between explicit rationals,
verifiable by integer arithmetic.

\subsection[The Baricz certificate at (2,3)]{The Baricz certificate at \((2,3)\)
(Remark~\ref{rem:baricz})}

For \(g(y)=g_{2,3}(y)=y(y+1)/((y+3)(y+4))\), the Jacobi witness
\eqref{eq:jacobi-witness} is
\(\rho_{2,3}(t)=te^{-3t}(-6+12e^{-t})\), negative exactly for \(t>\log2\).
Independently, exact differentiation gives
\begin{equation}
g'(y)=\frac{6\,(y^2+4y+2)}{(y+3)^2(y+4)^2},
\qquad
g''(y)=-\frac{12\,(y^3+6y^2+6y-10)}{(y+3)^3(y+4)^3},
\end{equation}
\begin{equation}
g'''(y)=\frac{36\,(y^4+8y^3+12y^2-40y-94)}{(y+3)^4(y+4)^4}.
\end{equation}
At \(y=1\) the quartic evaluates to \(1+8+12-40-94=-113\), so
\begin{equation}
g'''(1)=\frac{36\cdot(-113)}{4^4\cdot5^4}
=-\frac{4068}{160000}
=-\frac{1017}{40000}<0,
\end{equation}
confirming \(g\notin\BF\) by the one-point obstruction, in agreement with
Theorem~\ref{thm:baricz-classification}.

\subsection{The Berg--Pedersen table (Counterexample~\ref{cx:berg-pedersen})}

The source recurrence \cite{berg-pedersen-horn} is replayed exactly:
\begin{equation}
\rho_0=1,\qquad
\rho_n=\sum_{k=0}^{n-1}\rho_k\,
\frac{2\,(-1)^{n-1-k}}{(n-k+1)(n-k+2)}\quad(n\ge1),
\end{equation}
\begin{equation}
s_n=1+2\sum_{k=1}^{n}(-1)^k\rho_k,
\qquad
t_n=\sum_{k=0}^{n}(-1)^k\binom nk s_k ,
\end{equation}
and \(a_0=1\), \(a_n=1/t_n-1/t_{n-1}\). Exact rational arithmetic gives
\begin{equation}
\begin{array}{c|c|c}
n & t_n & a_n\\
\hline
0 & 1 & 1\\
1 & 2/3 & 1/2\\
2 & 5/9 & 3/10\\
3 & 67/135 & 72/335\\
4 & 371/810 & 4185/24857\\
5 & 1465/3402 & 75492/543515\\
6 & 209081/510300 & 36295938/306303665\\
7 & 85961/218700 & 1860116400/17972811841
\end{array}
\end{equation}
and
\begin{equation}
H_7(0)=a_0-7a_1+21a_2-35a_3+35a_4-21a_5+7a_6-a_7
=-\frac{28629387882812}{1395498975394445}.
\end{equation}

\subsection{The one-factor atom table (Proposition~\ref{prop:atoms})}

With \(\gamma_n=(1+\tfrac{2n}{5})(\tfrac23)^n\) (that is, \(q=2/3\),
\(c=2/5\)):
\begin{equation}
\gamma_0=1,\quad
\gamma_1=\frac{14}{15},\quad
\gamma_2=\frac45,\quad
\gamma_3=\frac{88}{135};
\qquad
\Delta^2\gamma_n=\Bigl(\frac23\Bigr)^{n}\frac{2n-3}{45},
\end{equation}
so \(\Delta^2\gamma_0=-1/15\) and \(\Delta^2\gamma_1=-2/135\), while
\(\Delta^2\gamma_n>0\) for \(n\ge2\). For \(c=1/100\) the first negative
entry of \eqref{eq:atom-differences} at \(n=0\) occurs at \(k=51\).

\section*{Declaration of competing interest}

The author declares no competing interests.

\section*{Funding}

The author received no specific funding for this work.

\section*{Data availability}

No empirical data were used. Every numerical claim reduces to finite exact
rational arithmetic, reproducible from the formulas in
Appendix~\ref{app:certificates}.

\section*{Declaration of generative AI and AI-assisted technologies in the
manuscript preparation process}

During the preparation of this work, the author used OpenAI ChatGPT and
OpenAI Codex to support literature triage, organization, language editing,
LaTeX formatting, consistency checks, and submission preparation. After using
these tools, the author reviewed and edited the content as needed, verified
the mathematical claims and cited sources, and takes full responsibility for
the content of the manuscript.

\end{document}